\documentclass[a4paper,11pt]{amsart}
\usepackage[
  top=3cm,
  bottom=3cm,
  left=3cm,
  right=3cm,
  headheight=14pt,
  headsep=1cm,
  footskip=1cm
]{geometry}
\usepackage{amsmath, epsfig, verbatim}
\usepackage{amsfonts}
\usepackage{amsthm}
\usepackage{mathtools}
\usepackage{amsmath}
\usepackage{amssymb} 
\usepackage{graphics,graphicx}
\usepackage{epstopdf}
\usepackage{listings}
\usepackage{xcolor}
\usepackage{tikz}
\usepackage{subcaption}
\usepackage{pgfplots}
\usepackage{caption}
\usepackage{booktabs}
\usepackage{tabularx}
\usepackage{hyperref} 
\usepackage{doi}
\usepackage{enumitem}
\usepackage[numbers,sort&compress]{natbib}
\usetikzlibrary{calc}
\allowdisplaybreaks
\definecolor{codegray}{rgb}{0.95,0.95,0.95}
\definecolor{pykeyword}{rgb}{0.13,0.13,1}
\definecolor{pystring}{rgb}{0.58,0,0.82}

\lstdefinestyle{pythonstyle}{
    backgroundcolor=\color{codegray},
    language=Python,
    basicstyle=\ttfamily\small,
    keywordstyle=\color{pykeyword}\bfseries,
    stringstyle=\color{pystring},
    commentstyle=\color{gray},
    showstringspaces=false,
    numbers=left,
    numberstyle=\tiny,
    frame=single,
    breaklines=true,
    tabsize=4,
}

\numberwithin{equation}{section}
\theoremstyle{plain}
\newtheorem{theorem}{Theorem}

\newtheorem{lemma}[theorem]{Lemma}

\begin{document}
\title{Visibility polynomial of some graph classes} 
\author{Tonny K B}
\address{Tonny K B, Department of Mathematics, College of Engineering Trivandrum, Thiruvananthapuram, Kerala, India, 695016.}
\email{tonnykbd@cet.ac.in}
\author{Shikhi M}
\address{Shikhi M, Department of Mathematics, College of Engineering Trivandrum, Thiruvananthapuram, Kerala, India, 695016.}
\email{shikhim@cet.ac.in}

\begin{abstract}
Mutual visibility in graphs provides a framework for analyzing how vertices can observe one another along shortest paths free of internal obstructions. The visibility polynomial, which enumerates mutual-visibility sets of all orders, has emerged as a central invariant in this study, with both theoretical significance and practical relevance in areas such as surveillance, target tracking, and distributed coordination. While computing this polynomial is computationally demanding, with known complexity $O(n^32^n)$, explicit characterizations for specific graph families yield valuable structural insights. In this paper, we characterize the mutual-visibility sets of several fundamental graph classes and derive closed-form expressions for their associated visibility polynomials. These results deepen the understanding of visibility-based invariants and expand the toolkit available for studying visibility phenomena in networks.
\end{abstract}

\maketitle
\section{Introduction}
 Let $G(V,E)$ be a simple graph and let $X\subseteq V$.  Two vertices $u,v\in V$ are said to be $X$-visible \cite{Stefano} if there exists a shortest path $P$ from $u$ to $v$ such that the internal vertices of $P$ do not belong to $X$; that is, $V(P)\cap X \subseteq \{u, v\}$. A set $X$ is called a mutual-visibility set of $G$ if every pair of vertices in $X$ is $X$-visible. The maximum size of such a set in $G$ is referred to as the mutual-visibility number, denoted by $\mu(G)$.

The study of mutual visibility in graphs has attracted considerable interest due to its wide-ranging applications in both theory and practice. The earliest investigations can be traced back to Wu and Rosenfeld, who examined visibility-related problems in pebble graphs~\cite{Geo_convex_1}. A more formal graph-theoretic definition of mutual-visibility sets was later given by Di Stefano~\cite{Stefano}. Since then, the notion of mutual visibility has emerged as an effective framework for understanding how information, influence, or coordination propagate within networks subject to topological restrictions. This line of research has been explored extensively in the literature~\cite{MV_2, MV_5, MV_9}, and several extensions and variants of the concept have also been introduced~\cite{MV_10}.

A related concept is the general position set, introduced independently in \cite{GP_1, GP_3}. This notion refers to a subset of vertices in a graph with the property that no three distinct vertices lie on the same geodesic. In other words, within such a set, no vertex occurs on a shortest path between any two of the others. Building on this idea, researchers have further developed the framework by defining the general position polynomial of a graph, which was proposed and studied in \cite{GP_2}.

 Examining mutual-visibility sets of different orders provides valuable insights into how groups of agents can maintain simultaneous observation of one another—a feature that is essential for applications such as surveillance, target tracking, and distributed coordination. When visibility is restricted by obstacles, studying sets of various sizes becomes important for identifying how many agents can remain in line-of-sight and adapt their positions accordingly. A significant advancement in this area was made by B. Csilla et al., who introduced the visibility polynomial in~\cite{sandi}. This polynomial invariant encodes the distribution of mutual-visibility sets of all orders in a graph. By enumerating these sets, one gains a richer understanding of the structural properties and visibility patterns inherent in the network.

 In \cite{VP_1}, the present authors studied the visibility polynomial of the join of two graphs. Also, it is identified that the algorithm for computing the visibility polynomial of a graph has a time complexity of $O(n^32^n)$, making the problem computationally intensive for larger graphs. In \cite{VP_2}, the present authors investigate the visibility polynomial associated with the corona product of two graphs. As part of this investigation, the concept of set-separator with respect to two subsets of the vertex set of a graph is defined. Let $u, v \in V(G)$. A vertex $g \in V(G)\setminus\{u, v\}$ is said to be a shortest-separator with respect to $u$ and $v$ if every shortest $(u,v)$-path contains $g$. The collection of all shortest-separators is called the path-cut of $G$, denoted by $p_c(G)$. Let $A$ and $B$ be two disjoint subsets of $V(G)$. A vertex $g$ is said to be a set-separator with respect to $A$ and $B$ if $g$ is a shortest-separator for every $u \in A$ and $v \in B$.

 Let $Q \subseteq V(G)$. A subset $W$ of $\overline{Q} = V(G) \setminus Q$ is said to be a $c_Q$-visible set \cite{VP_2} if $W$ is $Q$-visible  and $\{u, w\}$ is $Q$-visible for all $u \in Q$ and $w \in W$. A $c_Q$-visible set is said to be maximal if it is not a proper subset of any larger $c_Q$-visible set. Furthermore, a $c_Q$-visible set $W$ is called an absolute $c_Q$-visible set of $G$ if $Q$ is also a mutual-visibility set of $G$. That is, $W$ is an absolute $c_Q$-visible set if $Q \cup W$ is $Q$-visible. A maximal absolute $c_Q$-visible set is denoted by $\Omega_Q(G)$. The collection of all maximal absolute $c_Q$-visible sets is denoted by $\Gamma_Q(G)$.

 In this paper, the mutual-visibility sets of some fundamental graph classes are characterized, and explicit formulae for the visibility polynomial associated with the graph classes are derived.
 
\section{Notations and preliminaries}
  $G(V,E)$ represents an undirected simple graph with vertex set $V(G)$ and edge set $E(G)$. Unless otherwise stated, all graphs in this paper are assumed to be connected, so that there is at least one path between each pair of vertices. We follow the standard graph-theoretic definitions and notation as presented in \cite{Harary}.

   A complete graph on $n$ vertices is a graph in which there is an edge between any pair of distinct vertices, and is denoted by $K_n$. A sequence of vertices $(u_0,u_1,u_2,\ldots,u_{n})$ is referred to as a $(u_0,u_n)$-path in a graph $G$ if $u_iu_{i+1}\in E(G)$, $\forall i\in\{0,1,\ldots,(n-1)\}$. A cycle (or circuit) in a graph $G$ is a path $(u_0,u_1,u_2,\ldots, u_n)$ together with an edge $u_0u_n$. If a graph $G$ on $n$ vertices itself is a path, it is denoted by $P_n$, and if the graph $G$ itself is a cycle, it is denoted by $C_n$. For $n\ge 4$, the wheel graph $W_n = C_{n-1} \vee K_1 $. The unique vertex $h$ in $K_1$ is called the hub, and the cycle $C_{n-1}$ is called the rim of $W_n$. The helm graph $H_n$ is obtained from $W_n$ by attaching a pendant edge to each vertex of the rim $C_{n-1}$. Equivalently, $H_n = (W_n \odot K_1) \setminus e_h$, where $e_h$ denotes the unique pendant edge incident with the hub $h$ of $W_n$. The friendship graph $F_n$ is the graph formed by $n$ 3-cycles having exactly one vertex in common, called the center. A shell graph $S_n$, where $n \geq 3$, is obtained from the cycle $C_n$ by adding $(n-3)$ chords incident with a common vertex, called the apex $c$.  
Equivalently, $S_n = P_{n-1} \vee K_1$, where $P_{n-1}$ is a path on $n-1$ vertices and $K_1=\{c\}$ is the apex. The bow graph $B_{m,n}$ is defined as the union of two shell graphs $S_m$ and $S_n$ sharing the same apex $c$. Equivalently, $B_{m,n} = K_1 \vee \bigl(P_{m-1} \cup P_{n-1}\bigr)$, where $K_1=\{c\}$ denotes the apex, and $P_{m-1}$ and $P_{n-1}$ are disjoint paths on $m-1$ and $n-1$ vertices, respectively.

  The distance $ d_G(u,v)$ between two vertices $ u $ and $ v $ in $G$ is the length of the shortest $ (u,v)$-path in $ G$. Such a shortest path is referred to as a geodesic from $u$ to $v$.  The maximum distance between any pair of vertices of $G$  is called the diameter of  $G$, denoted by $diam(G)$. 

 Let $G$ and $H$ be two graphs. Then the join, $G \vee H$ is the graph with the vertex set $V(G)\cup V(H)$ and the edge set $E(G) \cup E(H)\cup \lbrace uv:u \in V(G), v\in V(H) \rbrace$.  The corona of $G$ and $H$, denoted by $G \odot H$, is the graph obtained by taking one copy of $G$ and $|V(G)|$ copies of $H$, and for each vertex $v$ in $G$, joining $v$ to each vertex in the corresponding copy of $H$ associated with  $v$, denoted by $H_v$.
\section{Visibility polynomial of some graph classes}
\begin{lemma}\label{P4.lem1}
For $n\ge 8$, let $W_n=C_{n-1}\vee K_1$ be the wheel graph with hub $h$ and rim $C_{n-1}$. Let $B\subseteq V(C_{n-1})$ with $|B|\ge 2$. Then $\{h\}\cup B$ is a mutual-visibility set of $W_n$ if and only if $B$ contains only two vertices $u$ and $v$ with $d_{C_{n-1}}(u,v)\le 2$.
\end{lemma}
\begin{proof}
For rim vertices $u$ and $v$,
\begin{equation}\label{P4.eq1}
d_{W_n}(u,v)=\min\{2,\,d_{C_{n-1}}(u,v)\},
\end{equation}
since $(u,h,v)$ is always a $(u,v)$-path of length two. Let $X=\{h\}\cup B$.

\smallskip
\noindent($\Rightarrow$) Suppose $X$ is a mutual-visibility set. If some $u,v\in B$ satisfy $d_{C_{n-1}}(u,v)\geq 3$, then by~\eqref{P4.eq1} the unique geodesic from $u$ to $v$ has length two and passes through $h\in X$. Hence $\{u,v\}$ is not $X$-visible, a contradiction. Therefore,
\begin{equation}\label{P4.eq2}
d_{C_{n-1}}(u,v) \leq 2 \quad \text{for all } u,v \in B.
\end{equation}

Suppose $|B|\geq 3$. Choose any $b \in B$, and let $b^+$ and $b^-$ denote the nearest vertices of $b$ in the clockwise and counterclockwise directions on $C_{n-1}$, respectively. Define $s_1 = d_{C_{n-1}}(b^-,b)$ and $s_2 = d_{C_{n-1}}(b,b^+)$. By~\eqref{P4.eq2}, we have $s_1,s_2 \in \{1,2\}$. The two $(b^-,b^+)$-paths on $C_{n-1}$ have lengths $s_1+s_2$ (which passes through $b$) and $(n-1)-(s_1+s_2)$. Thus,
\begin{equation}\label{P4.eq3}
d_{C_{n-1}}(b^-,b^+) = \min\{s_1+s_2,\,(n-1)-(s_1+s_2)\} \;\geq\; 2,
\end{equation}
since $2 \leq s_1+s_2 \leq 4$ and $(n-1)-(s_1+s_2)\geq 3$. Hence, by~\eqref{P4.eq2}, we must have $d_{C_{n-1}}(b^-,b^+)=2$. Therefore, from~\eqref{P4.eq3}, it follows that $s_1+s_2=2$, as $(n-1)-(s_1+s_2)\geq 3$. Consequently, the shortest $(b^-,b^+)$-path along the rim is $(b^-,b,b^+)$, and moreover $d_{W_n}(b^-,b^+)=2$. Thus every shortest $(b^-,b^+)$-path contains either $b$ or $h$, implying that $B$ is not $X$-visible, a contradiction. Therefore $B=\{u,v\}$ with $d_{C_{n-1}}(u,v)\leq 2$.

\smallskip
\noindent($\Leftarrow$) Conversely, let $B=\{u,v\}$ with $d_{C_{n-1}}(u,v)\leq 2$.  
The pairs $\{h,u\}$ and $\{h,v\}$ are $X$-visible via the edges $hu$ and $hv$. If $d_{C_{n-1}}(u,v)=1$, then the edge $uv$ is a geodesic from $u$ to $v$, with no internal vertices. If $d_{C_{n-1}}(u,v)=2$, then there exists a shortest $(u,v)$-path $(u,w,v)$ in $C_{n-1}$, whose unique internal vertex $w$ lies on the rim. Since $|B|=2$, we have $w\notin X$. Hence, in both cases $\{u,v\}$ is $X$-visible, and therefore $X$ is a mutual-visibility set. This completes the proof.
\end{proof}

\begin{theorem}
For $n\geq 8$, the visibility polynomial of the wheel graph $W_n$ is given by, $\mathcal{V}(W_n)=(1+x)^{n-1}+x+(n-1)x^2+2(n-1)x^3$.

\end{theorem}

\begin{proof}
Let $S \subseteq V(C_{n-1})$ and $u,v \in S$. If $u$ and $v$ are adjacent, then they are $S$-visible trivially. If they are not adjacent, then $d_{C_{n-1}}(u,v) \geq 2$. In this case, there exists a shortest $(u,v)$-path $P = (u,h,v)$ through the hub $h$ that avoids $S$ internally. Therefore, every subset of the rim $C_{n-1}$ is a mutual-visibility set of $W_n$. Hence, the corresponding contribution to $\mathcal{V}(W_n)$ is $(1+x)^{n-1}$.

Now, let $S = \{h\} \cup B$ with $B \subseteq V(C_{n-1})$. If $S$ is a mutual-visibility set, then by Lemma~\ref{P4.lem1}, we have $|B| \leq 2$. If $|B| = 0$, then $S = \{h\}$, contributing $x$. If $|B| = 1$, then for each rim vertex $u$, the set $\{h,u\}$ is a mutual-visibility set, contributing $(n-1)x^2$. If $|B| = 2$, say $B = \{u,v\}$, then by Lemma~\ref{P4.lem1}, $S$ is a mutual-visibility set if and only if $d_{C_{n-1}}(u,v) \leq 2$. In $C_{n-1}$, there are exactly $(n-1)$ adjacent pairs and $(n-1)$ pairs at distance $2$, yielding $2(n-1)$ such sets, contributing $2(n-1)x^3$ to $\mathcal{V}(W_n)$. This completes the proof.
\end{proof}

\begin{lemma}\label{P4.lem2}
Let $Q$ be a mutual-visibility set of the wheel graph $W_n$, $n \geq 8$, with hub $h$. If $h \notin Q$, then $\overline{Q}$ is the unique maximal absolute $c_Q$-visible set; hence, $Q$ is disjoint-visible.
\end{lemma}
\begin{proof}
Let $x,y \in \overline{Q}$. If $h \in \{x,y\}$ and, if $y$ lies on the rim, then the edge $hy$ is the shortest $(x,y)$-path, having no internal vertices and thus avoiding $Q$. If $h \notin \{x,y\}$, then $x$ and $y$ are both rim vertices. By the equation ~\eqref{P4.eq1}, 
\[
d_{W_n}(x,y) = \min\{\,2,\, d_{C_{n-1}}(x,y)\,\}.
\]
If $x$ and $y$ are adjacent on the rim, then the edge $xy$ is the unique geodesic between them, with no internal vertices. Otherwise, the path $(x,h,y)$ is a shortest $(x,y)$-path whose unique internal vertex is $h \notin Q$. Therefore, $\overline{Q}$ is $Q$-visible.

Now, let $x \in Q$ and $y \in \overline{Q}$. If $y = h$, then the edge $xh$ is a shortest path with no internal vertices. Otherwise, $x$ and $y$ are both rim vertices, and as established in the previous case, they are $Q$-visible.

Thus, $\overline{Q}$ is $c_Q$-visible and, since $Q$ is a mutual-visibility set, it follows that $\overline{Q}$ is the unique maximal absolute $c_Q$-visible set. Consequently, $Q$ is disjoint-visible.
\end{proof}
\begin{lemma}[{\cite[Lem. 1]{VP_2}}]\label{P2.lem1}
    Let $A \subseteq V(G)$. Then $A$ is a mutual-visibility set of $G$ if and only if it is a mutual-visibility set of $G \odot H$.
\end{lemma}
\begin{lemma}[{\cite[Lem. 2]{VP_2}}]\label{P2.lem2}
 Let $G$ and $H$ be two graphs. If $S \subseteq \cup_{w \in V(G)} V(H_w) $, then $S$ is a mutual-visibility set of $G \odot H$. 
\end{lemma}
\begin{lemma}[{\cite[Lem. 3]{VP_2}}]\label{P2.lem3}
 If $v \in V(G)$ and $\emptyset \subsetneq B\subseteq \cup_{w \in V(G) } V(H_w)$, then the set $\{ v\} \cup B$ is a mutual-visibility set of $G \odot H$ if and only if $B$ satisfies one of the following conditions.\\
    1. $B$ is a mutual-visibility set of $H$ with $diam_{H}(B) \leq 2$.\\
    2. $B \subseteq \cup_{w \in V(G)\setminus \{v\} } V(H_w)$ and $\{v\}$-visible subset of $G \odot H$.
\end{lemma}

\begin{lemma}[{\cite[Lem. 22]{VP_2}}]\label{P2.lem12}
Let $G$ be a graph and $\emptyset \neq Q \subsetneq V(G)$, and $S=Q \cup S_{\overline{Q}}$, where $\emptyset \neq S_{\overline{Q}} \subseteq \cup_{w \in \overline{Q} } V(H_w)$, is a mutual-visibility set of $G \odot H$. The contribution to the visibility polynomial of $G \odot H$ corresponding to mutual-visibility sets of the form $S$ is $\sum_{Q} p_Q(x)$, where the sum is taken over all proper mutual-visibility sets $Q$ of $G$ and 
$$p_Q(x) = \left\{
  \begin{array}{ll}
  \sum_{\Omega_Q(G)}\left((1+x)^{|\Omega_Q(G)||V(H)|}-1\right)x^{|Q|}& \text{if } Q \text{ is disjoint-visible} \\[10pt]
  \sum_{\emptyset \neq \mathcal{J} \subseteq \Gamma_Q(G)} (-1)^{|\mathcal{J}|+1} \left((1+x)^{\left| \cap_{W \in \mathcal{J}} W \right||V(H)|}-1\right)x^{|Q|} & \text{Otherwise}
  \end{array}
\right.$$
\end{lemma}
\begin{lemma}[{\cite[Lem. 4]{VP_2}}]\label{P2.lem5}
  Let $A \subseteq V(G)$ where $|A| \geq 2$ and let $\ B\subseteq \cup_{w \in V(G) } V(H_w)$. If $B$ contains at least one vertex of $\cup_{a \in A } H_a $, then $A \cup B$ is not a mutual-visibility set of $G \odot H$.
\end{lemma}
\begin{theorem}
    The visibility polynomial of helm graph $H_n,\ n \geq 8$ is given by 
    $$\mathcal{V}(H_n)= \mathcal{V}(W_n)+((1+x)^{n-1}-1)+(n-1)x^2+\sum_Q w_Q(x)$$
    where, 
   $$w_Q(x) = \left\{
  \begin{array}{ll}
  \left((1+x)^{|\overline{Q}|-1}\right)x^{|Q|}& \text{if } h \notin Q \\[10pt]
   p_Q(x) & \text{Otherwise}
  \end{array}
\right.$$
\end{theorem}
\begin{proof}
By definition, $H_n = (W_n \odot K_1)\setminus e_h$, where $e_h$ denotes the pendant edge incident with the hub $h$ in $W_n$. All mutual-visibility sets of $G \odot H$ are characterized in Lemmas~\ref{P2.lem1}, \ref{P2.lem2}, \ref{P2.lem3}, and \ref{P2.lem12}. Accordingly, we compute $\mathcal{V}(H_n)$ by considering five distinct cases. By Lemma~\ref{P2.lem1}, all mutual-visibility sets of $W_n$ are also mutual-visibility sets of $H_n$. Therefore, the contribution to $\mathcal{V}(H_n)$ in this case is $\mathcal{V}(W_n)$. By Lemma~\ref{P2.lem2}, all subsets of $\cup_{w \in V(W_n) \setminus \{h\}} V((K_1)_w)$ are also mutual-visibility sets of $H_n$; however, the empty subset is already counted in the previous case. Hence, the contribution in this case is $\left((1+x)^{n-1} - 1\right)$. Next, we compute the contribution corresponding to Case~1 of Lemma~\ref{P2.lem3}. Specifically, mutual-visibility sets of the form $\{v\} \cup B$, where $v \in V(W_n) \setminus \{h\}$ and $B$ is a mutual-visibility set of $K_1$, with diameter at most two, contribute $x(\mathcal{V}_{2}(K_1)-1)=x(1+x-1)$. Since there are $n-1$ such vertices $v$, the total contribution in this case is $(n-1)x^2$.

    Finally, Lemma~\ref{P2.lem12} accounts for the remaining contribution to $\mathcal{V}(H_n)$, which includes the mutual-visibility sets referred to in the second case of Lemma~\ref{P2.lem3}. Specifically, this corresponds to mutual-visibility sets of the form $S = Q \cup S_{\overline{Q}}$, where $\emptyset \neq Q \subsetneq V(W_n)$ and $\emptyset \neq S_{\overline{Q}} \subseteq \cup_{w \in \overline{Q}\setminus \{h\}} V((K_1)_w)$. By Lemma~\ref{P4.lem2}, if $h \notin Q$, then $\overline{Q}$ is the unique maximal absolute $c_Q$-visible set, which contains $h$. So the corresponding contribution to $\mathcal{V}(H_n)$ is $\left((1+x)^{|\overline{Q}|-1}\right)x^{|Q|}$. By Lemma~\ref{P2.lem12}, if $h \in Q$, then the corresponding contribution to $\mathcal{V}(H_n)$ is $p_Q(x)$, where 
    $$p_Q(x) = \left\{
  \begin{array}{ll}
  \sum_{\Omega_Q(G)}\left((1+x)^{|\Omega_Q(G)|}-1\right)x^{|Q|}& \text{if } Q \text{ is disjoint-visible} \\[10pt]
  \sum_{\emptyset \neq \mathcal{J} \subseteq \Gamma_Q(G)} (-1)^{|\mathcal{J}|+1} \left((1+x)^{\left| \cap_{W \in \mathcal{J}} W \right|}-1\right)x^{|Q|} & \text{Otherwise}
  \end{array}
\right.$$
    Note that, by Lemma~\ref{P2.lem5}, there is no mutual-visibility set of the form $S = V(G) \cup S_{\overline{Q}}$. This completes the proof.
\end{proof}

\begin{lemma}\label{P4.lem7}
Let \(F_n\) be the friendship graph with center \(c\) and 3-cycles 
\(T_i=(c,a_i,b_i)\) for \(i=1,\dots,n\). 
A set \(S\subseteq V(F_n)\) is a mutual-visibility set if and only if $S$ satisfies one of the following condition
\begin{enumerate}
\item \(c\notin S\)
\item \(c\in S\) and \(S\setminus\{c\}\subseteq\{a_i,b_i\}\) for some \(i\).
\end{enumerate}
\end{lemma}

\begin{proof}
\noindent($\Rightarrow$) Suppose $S$ is a mutual-visibility set of $F_n$ with $c\in S$. 
If there exist $i\neq j$ with $u\in\{a_i,b_i\}\cap S$ and $v\in\{a_j,b_j\}\cap S$, then the unique geodesic between $u$ and $v$ is $(u,c,v)$, whose internal vertex $c$ lies in $S$, contradicting the mutual-visibility of $S$. Therefore, when $c\in S$, all vertices of $S\setminus\{c\}$ must lie within a single 3-cycle.

\noindent($\Leftarrow$) If $c\notin S$, then for any $u,v\in S$ either $u$ and $v$ are adjacent, or the unique geodesic is $(u,c,v)$ with internal vertex $c\notin S$. In both cases $u$ and $v$ are $S$-visible. If $c\in S$ and $S\setminus\{c\}\subseteq\{a_i,b_i\}$ for some $i$, then the only possible pairs to check are $(a_i,b_i)$, $(c,a_i)$ and  $(c,b_i)$. Since there are edges between the vertices of each of these pairs, $S$ is a mutual-visibility set. This completes the proof.
\end{proof}

\begin{theorem}
Let $F_n$ be the friendship graph with center $c$ and 3-cycles $T_i=(c,a_i,b_i)$ for $i=1,\dots,n$. Its visibility polynomial is $\mathcal{V}(F_n)= (1+x)^{2n} + x + 2n x^2 + n x^3$.
\end{theorem}
\begin{proof}
Let $S$ be a mutual-visibility set of $F_n$. By Lemma~\ref{P4.lem7}, either 
$c\notin S$, or $c\in S$ and $S\setminus\{c\}\subseteq\{a_i,b_i\}$ for some $i$.

\medskip
\noindent\textbf{Case 1}: If $c\notin S$, then $S$ may be any subset of the 
$2n$ noncentral vertices. The contribution of these mutual-visibility sets to $\mathcal{V}(F_n)$ is given by $(1+x)^{2n}$.

\medskip
\noindent\textbf{Case 2}: If $c\in S$ and $S\setminus\{c\}\subseteq\{a_i,b_i\}$ 
for some $i$, then $S\setminus\{c\}$ can only be one of the following four subsets: $\emptyset$, $\{a_i\}$, $\{b_i\}$, or $\{a_i,b_i\}$. These yield sets of 
sizes $1,2,2,$ and $3$, respectively. The singleton $\{c\}$ contributes $x$ to $\mathcal{V}(F_n)$. For each $i$, there are two sets $\{c,a_i\},\{c,b_i\}$ of size 2 and one set $\{c,a_i,b_i\}$ of size 3 which correspond to the 3-cycle $T_i$. Altogether, the $n$ 3-cycles contribute $2n x^2 + n x^3$ to $\mathcal{V}(F_n)$. 

Adding the contributions from the two cases, we obtain $\mathcal{V}(F_n)= (1+x)^{2n} + x + 2n x^2 + n x^3$.
\end{proof}

\begin{lemma}\label{P4.lem8}
Let $S_n = P_{n-1} \vee K_1$ be the shell graph on $n \geq 3$ vertices, with apex $c$. For $S = \{c\} \cup A$, where $A \subseteq V(P_{n-1})$, the set $S$ is a mutual-visibility set of $S_n$ if and only if $|A| \leq 2$ and $diam(P_{n-1}[A]) \leq 2$.
\end{lemma}
\begin{proof}
Let $P_{n-1}= (v_1, v_2, \ldots, v_{n-1})$. Since, $c$ is adjacent to every $v_i$, distances in $S_n$ satisfy:
\[
d_{S_n}(v_i,v_j)=
\begin{cases}
1 \quad \text{if }&|i-j|=1\\
2 \quad \text{if }&|i-j|\geq 2
\end{cases}
\quad \text{and} \quad
d(c,v_i)=1.
\]
\smallskip
\noindent($\Rightarrow$) Assume that $S=\{c\}\cup A$ is a mutual-visibility set of $S_n$. We first claim that $A$ cannot contain three consecutive vertices from $P_{n-1}$. Indeed, if $v_{i-1},v_i,v_{i+1}\in A$, then for the pair $\{v_{i-1},v_{i+1}\}$ the two geodesics 
$(v_{i-1},c,v_{i+1})$ and $(v_{i-1},v_i,v_{i+1})$ both have their unique internal vertex in $S$ 
(namely $c$ and $v_i$, respectively), contradicting the mutual-visibility of $S$. 

Next, consider a pair $\{v_i,v_j\}$ with $|i-j|\ge 3$. 
The unique geodesic between them in $S_n$ is $(v_i,c,v_j)$, whose internal vertex $c$ lies in $S$. 
Thus such pairs of vertices cannot belong to $A$. These two restrictions imply that $|A|\le 2$, and any pair in $A$ has distance at most $2$ in $P_{n-1}$. 
Equivalently, $|A|\le 2$ and $diam(P_{n-1}[A])\le 2$.

\smallskip
\noindent($\Leftarrow$)Conversely, suppose that $|A|\le 2$ and $diam(P_{n-1}[A]) \le 2$. We will prove that $S=\{c\}\cup A$ is $S$-visible. Pairs of the form $(c,v_i)$ with $v_i\in A$ are adjacent, and hence $S$-visible. If $A=\emptyset$ or $|A|=1$, there is nothing further to check. If $|A|=2$, say $A=\{v_i,v_j\}$, then $1\le |i-j|\le 2$ by the diameter condition on $A$. If $|i-j|=1$, the edge $v_iv_j$ is a geodesic with no internal vertices. 
If $|i-j|=2$, then $v_i$ and $v_j$ are joined by exactly two geodesics of length $2$, namely $(v_i,c,v_j)$ and $(v_i,v_{i+1},v_j)$. Since $|A|=2$, the middle vertex $v_{i+1}\notin S$; hence the internal vertex of the latter geodesic does not belong to $S$. Thus, in every case, a shortest $(u,v)$-path with all internal vertices outside $S$ exists for each pair $u,v\in S$. Therefore, $S$ is a mutual-visibility set of $S_n$.
\end{proof}
\begin{theorem}\label{P4.th3}
Let $S_n$ be the shell graph on $n\geq 3$ vertices.  
Then the visibility polynomial of $S_n$ is given by
\[
\mathcal{V}(S_n) = (1+x)^{n-1} + x + (n-1)x^2 + (2n-5)x^3.
\]
\end{theorem}
\begin{proof}
Note that $S_n = P_{n-1}\vee K_1$, where $P_{n-1}= (v_1, v_2, \ldots, v_{n-1})$. Label the apex by $c$. Let $S$ be a subset of $V(S_n)$ and  suppose that $c\notin S$. If $v_i$ and $v_j$ are not adjacent ($|i-j|\geq 2$) then, $(v_i, c, v_j)$ is a geodesic from $v_i$ to $v_j$, whose internal vertex $c\notin S$. Hence, every subset of $V(P_{n-1})$ is a mutual-visibility set of $S_n$. Therefore, the contribution of such mutual-visibility sets to $\mathcal{V}(S_n)$ is $\sum_{k=0}^{n-1}\binom{n-1}{k}x^k=(1+x)^{n-1}$.

Now suppose $c\in S$ and write $S=\{c\}\cup A$ with $A\subseteq V(P_{n-1})$. By Lemma~\ref{P4.lem8}, $S$ is a mutual-visibility set of $S_n$ if and only if $|A|\le 2$ and $diam(P_{n-1}[A])\le 2$. Hence, the admissible choices for $A$ are:
\begin{itemize}
\item $|A|=0$: $A=\emptyset$, giving one choice and contributing $x$;
\item $|A|=1$: $A$ contains exactly one vertex of $P_{n-1}$, yielding $n-1$ choices and contributing $(n-1)x^2$;
\item $|A|=2$: $A$ contains two vertices of $P_{n-1}$ at distance at most $2$. 
There are $(n-2)$ adjacent pairs $\{v_i,v_{i+1}\}$ and $(n-3)$ pairs at a distance $2$, namely $\{v_i,v_{i+2}\}$, 
for a total of $(2n-5)$ choices, contributing $(2n-5)x^3$.
\end{itemize}

Therefore, the total contribution from sets with $c\in S$ is $x+(n-1)x^2+(2n-5)x^3$. Combining this with the contribution from sets with $c\notin S$, namely $(1+x)^{n-1}$, we obtain $\mathcal{V}(S_n)=(1+x)^{n-1}+x+(n-1)x^2+(2n-5)x^3$.
\end{proof}

\begin{lemma}[{\cite[Lem. 7]{VP_2}}]\label{P2.lem7}
 Let $g$ be a set-separator with respect to the disjoint sets $A$ and $B$ of $V(G)$. If $S$ is a mutual-visibility set of $G$, containing $g$, then either $S \cap A =\emptyset$ or $S \cap B =\emptyset$. That is, either $S\subseteq \overline{A}$ or $S\subseteq \overline{B}$.
\end{lemma}
\begin{theorem}\label{P4.th4}
Let $B_{m,n}$ be the bow graph obtained as the union of two shell graphs $S_m$ and $S_n$ sharing the same apex $c$.  
Then the visibility polynomial of $B_{m,n}$ is given by
\[
\mathcal{V}(B_{m,n}) = (1+x)^{m+n-2} + x + (m+n-2)x^2 + (2m+2n-10)x^3.
\]
\end{theorem}
\begin{proof}
By definition, $B_{m,n} = K_1 \vee \bigl(P_{m-1} \cup P_{n-1}\bigr)$, where $K_1=\{c\}$ is the apex and  
$P_{m-1}=(a_1,\ldots,a_{m-1})$, $P_{n-1}=(b_1,\ldots,b_{n-1})$ are disjoint paths. Let $S \subseteq V(B_{m,n})$. To compute the number of mutual-visibility sets of $B_{m,n}$, we consider two cases.

\smallskip
\noindent\textbf{Case 1}: Suppose $c\notin S$. For any $u,v\in V(P_{m-1})$ (or both in $V(P_{n-1})$), either $uv$ is an edge of the path, or a shortest $(u,v)$-path is $(u,c,v)$ of length $2$, whose internal vertex $c\notin S$. For any $u\in V(P_{m-1})$ and $v\in V(P_{n-1})$, the unique geodesic between $u$ and $v$ is $(u,c,v)$, again with internal vertex $c\notin S$. Hence every subset of $V(P_{m-1})\cup V(P_{n-1})$ is a mutual-visibility set. The contribution of such sets to $\mathcal{V}(B_{m,n})$ is
\[
\sum_{k=0}^{m+n-2}\binom{m+n-2}{k}x^k=(1+x)^{m+n-2}.
\]

\smallskip
\noindent\textbf{Case 2}: Suppose $c\in S$. 
Here $c$ is a set-separator with respect to $V(P_{m-1})$ and $V(P_{n-1})$. By Lemma~\ref{P2.lem7}, if $S$ is a mutual-visibility set, then $S$ cannot meet both $V(P_{m-1})$ and $V(P_{n-1})$. Hence either $S=\{c\}\cup A$ with $A\subseteq V(P_{m-1})$, or $S=\{c\}\cup B$ with $B\subseteq V(P_{n-1})$. From the argument in the proof of Theorem~\ref{P4.th3}, the contribution of mutual-visibility sets of the form $S=\{c\}\cup A$ is $x+(m-1)x^2+(2m-5)x^3$. Similarly, the contribution of those of the form $S=\{c\}\cup B$ is $x+(n-1)x^2+(2n-5)x^3$. Since the singleton $\{c\}$ is counted in both cases, we include it only once. 
Hence the total contribution to $\mathcal{V}(B_{m,n})$ with $c\in S$ is
\[
x+(m-1)x^2+(2m-5)x^3 + (n-1)x^2+(2n-5)x^3 
= x+(m+n-2)x^2+(2m+2n-10)x^3.
\]
Combining with Case 1 gives the stated formula.
\end{proof}
\section{Conclusion}
In the paper, closed-form expressions for the visibility polynomials of the wheel, helm, friendship, shell, and bow graphs are presented. These results provide explicit characterizations of mutual-visibility sets in fundamental graph classes, offering both structural insights and computational advantages. The findings broaden the scope of visibility-based invariants and lay the groundwork for further exploration in more complex networks and applications.
\bibliographystyle{plainurl}
\bibliography{cas-refs}
\end{document}